\documentclass[12pt]{amsart}
\usepackage{latexsym, amsthm, amscd, euscript, float, amssymb, amsmath, graphicx,ifthen,cite}
\usepackage{hypgotoe,hyphsubst}

\newcounter{alphabet}
\newcounter{tmp}
\newenvironment{Thm}[1][]{\refstepcounter{alphabet}%
\bigskip%
\noindent%
{\bf Theorem \Alph{alphabet}}%
\ifthenelse{\equal{#1}{}}{}{ (#1)}%
{\bf .} \itshape}{\vskip 8pt}

\makeatletter
\newcommand{\Ref}[1]{\@ifundefined{r@#1}{}{\setcounter{tmp}{\ref{#1}}\Alph{tmp}}}
\makeatother

\newcounter{alphabet2}

\newtheorem{theorem}{Theorem}[section]
\newtheorem{lemma}[theorem]{Lemma}

\newtheorem{ca}{Case}

\newtheorem{remark}[theorem]{Remark}
\newtheorem{example}[theorem]{Example}

\numberwithin{equation}{section}

\newcounter{minutes}
\divide\time by 60
\newcounter{hours}
\multiply\time by 60
\addtocounter{minutes}{-\time}

\newcommand{\IR}{\mathbb{R}}

\newcommand{\ds}{\displaystyle}

\newcommand{\R}{\mathbb{R}}
\newcommand{\C}{\mathbb{C}}

\newcommand{\real}{{\rm Re}}
\def\be{\begin{equation}}
\def\ee{\end{equation}}

\newcommand{\bca}{\begin{ca}}
\newcommand{\eca}{\end{ca}}
\newcommand{\bsca}{\begin{sca}}
\newcommand{\esca}{\end{sca}}

\newcommand{\bcl}{\begin{cl}}
\newcommand{\ecl}{\end{cl}}

\newcommand{\bscl}{\begin{scl}}
\newcommand{\escl}{\end{scl}}

\newcommand{\bcons}{\begin{conjs}}
\newcommand{\econs}{\end{conjs}}
\newcommand{\bprop}{\begin{propo}}
\newcommand{\eprop}{\end{propo}}
\newcommand{\br}{\begin{rem}}
\newcommand{\er}{\end{rem}}
\newcommand{\brs}{\begin{rems}}
\newcommand{\ers}{\end{rems}}
\newcommand{\bo}{\begin{obser}}
\newcommand{\eo}{\end{obser}}
\newcommand{\bos}{\begin{obsers}}
\newcommand{\eos}{\end{obsers}}
\newcommand{\ba}{\begin{array}}
\newcommand{\ea}{\end{array}}
\newcommand{\beq}{\begin{eqnarray}}
\newcommand{\beqq}{\begin{eqnarray*}}
\newcommand{\eeq}{\end{eqnarray}}
\newcommand{\eeqq}{\end{eqnarray*}}

\begin{document}
\vspace*{-2cm}
\title[Boundary Schwarz Lemma for non-homogeneous biharmonic equations]
{Boundary Schwarz lemma for solutions to non-homogeneous biharmonic equations}
\def\thefootnote{}
\footnotetext{ \texttt{\tiny File:~\jobname .tex,
          printed: \number\day-\number\month-\number\year,
          \thehours.\ifnum\theminutes<10{0}\fi\theminutes}
} \makeatletter\def\thefootnote{\@arabic\c@footnote}\makeatother

\author[M. R. Mohapatra]{Manas Ranjan Mohapatra}
\address{Manas Ranjan Mohapatra, Department of Mathematics,
Shantou University, Shantou, 515063, People's Republic of China
}
\email{manas@stu.edu.cn}

\author[X. T. Wang]{Xiantao Wang${}^{~\mathbf{*}}$}
\address{Xiantao Wang, Department of Mathematics, Hunan Normal University, Changsha, Hunan 410081, People's Republic of China, and Department of Mathematics,
Shantou University, Shantou, Guangdong 515063, People's Republic of China
}
\email{xtwang@hunnu.edu.cn}

\author[J.-F. Zhu]{Jian-Feng Zhu}
\address{Jian-Feng Zhu, Department of Mathematics,
	Shantou University, Shantou, 515063, People's Republic of China
	and School of Mathematical Sciences, Huaqiao University, Quanzhou, 362021
	, People's Republic of China
}
\email{flandy@hqu.edu.cn}

\begin{abstract}
In this paper, we establish a boundary Schwarz lemma for solutions to
non-homogeneous biharmonic equations. 
\end{abstract}

\noindent

\subjclass[2000]{Primary: 30C80; Secondary: 31A30}
\keywords{Boundary Schwarz lemma; solution; non-homogeneous biharmonic equation.\\
$^{\mathbf{*}}$Corresponding author}

\keywords{
Boundary Schwarz lemma; solution; non-homogeneous biharmonic equation.\\
$^{\mathbf{*}}$Corresponding author}

\maketitle
\thispagestyle{empty}
\section{Introduction and Main Result}
The classical Schwarz lemma says that an analytic function
$f$ from the unit disk $\mathbb{D}=\{z\in
\mathbb{C}:|z|<1\}$ into itself with $f(0)=0$ must map
each smaller disk $\{z\in \mathbb{C}:\; |z|<r<1\}$ into itself. Also, $|f'(0)|\le 1$, and $|f'(0)|= 1$
if and only if $f$ is a rotation of $\mathbb{D}$.
This is a very powerful tool in complex analysis.
An elementary consequence of Schwarz lemma is that if $f$ extends
continuously to some boundary point $\alpha$,
$|f(\alpha)|=1$, and if $f$ is differentiable at $\alpha$, then $|f'(\alpha)|\ge 1$ (see, for example,
\cite{Gar-book,Oss00}).

Establishing various versions of Schwarz lemma
and boundary Schwarz lemma has
attracted many researchers in recent years. In \cite{BK94}, Burns and Krantz obtained a Schwarz lemma at the boundary for holomorphic mappings defined on $\mathbb{D}$ as well as on balls in $\mathbb{C}^n$. They have also obtained similar results for holomorphic mappings on strongly convex and strongly pseudoconvex domains in $\mathbb{C}^n$. Liu and Tang in \cite{LT15} obtained the boundary Schwarz lemma for holomorphic mappings defined on the unit ball in $\mathbb{C}^n$. We refer the survey article by Krantz \cite{Kra11} for a brief history on the Schwarz lemma at the boundary.

The Schwarz lemma at the boundary plays an important
role in complex analysis. For example, by using the Schwarz lemma at the boundary,
Bonk improved the previously known lower bound for the Bloch constant in \cite{Bonk90}.
The boundary Schwarz lemma is also a fundamental tool in the study of the geometric properties of functions of several complex variables; see \cite{LT15,LT16,LWT15}.
In this paper, we are interested in establishing a boundary Schwarz lemma for functions which satisfy certain partial differential equations, namely,
the non-homogeneous biharmonic equations.
We now proceed to write some notations and preliminaries which are required
to state our result.

We denote by $\mathbb{T}=\partial\mathbb{D}$ the boundary of
$\mathbb{D}$, and by $\overline{\mathbb{D}}=\mathbb{D}\cup \mathbb{T}$, the closure of $\mathbb{D}$. For any subset $\Omega$ of $\C$, we denote by $\mathcal{C}^{m}(\Omega)$, the set of all complex-valued $m$-times continuously differentiable functions from
$\Omega$ into $\mathbb{C}$, where $m\in \mathbb{N}\cup\{0\}$. In
particular, $\mathcal{C}(\Omega):=\mathcal{C}^{0}(\Omega)$ denotes the
set of all continuous functions in $\Omega$.

For a real $2\times2$ matrix $A$, we use the matrix norm
$$\|A\|=\sup\{|Az|:\,z\in \mathbb{T}\}$$ and the matrix function
$$\lambda(A)=\inf\{|Az|:\, z\in \mathbb{T}\}.$$

For $z=x+iy\in\mathbb{C}$ with $x$, $y\in \IR$, the
formal derivative of a complex-valued function $f=u+iv$ is given
by
$$D_{f}=\left(\begin{array}{cccc}
\ds u_{x}\;~~ u_{y}\\[2mm]
\ds v_{x}\;~~ v_{y}
\end{array}\right),
$$
so that
$$\|D_{f}\|=|f_{z}|+|f_{\overline{z}}| ~\mbox{ and }~ \lambda(D_{f})=\big| |f_{z}|-|f_{\overline{z}}|\big |,
$$
where $$f_{z}=\frac{1}{2}\big(
f_x-if_y\big)\;\;\mbox{and}\;\; f_{\overline{z}}=\frac{1}{2}\big(f_x+if_y\big).$$  We use
$$J_{f}:=\det D_{f} =|f_{z}|^{2}-|f_{\overline{z}}|^{2}
$$
to denote the {\it Jacobian} of $f$.

Let $f^{\ast}, g\in\mathcal{C}(\overline{\mathbb{D}})$, $\varphi \in \mathcal{C}(\mathbb{T})$
 and
$f\in\mathcal{C}^{4}(\mathbb{D})$. We are interested in
the following {\it non-homogeneous biharmonic equation} defined in
$\mathbb{D}$:
\be\label{eq-ch-1.0}\Delta(\Delta f)=g\ee with the following
associated Dirichlet boundary value: \be\label{eq-ch-1}
\begin{cases}
\displaystyle  f_{\overline{z}}=\varphi &\mbox{ on }\, \mathbb{T},\\
\displaystyle f=f^{\ast}&\mbox{ on }\, \mathbb{T},
\end{cases}
\ee  where $$\Delta f=f_{xx}+f_{yy}=4f_{z \overline{z}}$$ is the
{\it Laplacian} of $f$.

In particular, if $g\equiv0$, then any
solution to (\ref{eq-ch-1.0}) is {\it biharmonic}. For the
properties of biharmonic mappings, see
\cite{CWY99,Str03}.
Chen et. al. in \cite{CLW} have discussed the Schwarz-type lemma, Landau-type theorems and bilipschitz properties for the solutions of
non-homogeneous biharmonic equations \eqref{eq-ch-1.0} satisfying \eqref{eq-ch-1}.

Suppose that
\be\label{G-1}G(z,w)=|z-w|^{2}\log\left|\frac{1-z\overline{w}}{z-w}\right|^{2}-(1-|z|^{2})(1-|w|^{2})\ee
and
$$P(z,e^{i\theta})=\frac{1-|z|^{2}}{|1-ze^{-i\theta}|^{2}} \quad (\theta\in[0,2\pi])
$$
denote the {\it biharmonic Green function} and {\it (harmonic)
Poisson kernel} in $\mathbb{D}$, respectively. It follows from  \cite[Theorem 2]{Beg05} that all the
solutions to the equation (\ref{eq-ch-1.0}) satisfying the boundary conditions (\ref{eq-ch-1}) are given by
\beq\label{eq-ch-3}
f(z)&=&\mathcal{P}_{f^{\ast}}(z)+\frac{1}{2\pi}(1-|z|^{2})\int_{0}^{2\pi}f^{\ast}(e^{it})\frac{\overline{z}e^{it}}{(1-\overline{z}e^{it})^{2}}dt\\
\nonumber
&&-(1-|z|^{2})\mathcal{P}_{\varphi_{1}}(z)-\frac{1}{8}G[g](z),
\eeq
where
\be\label{mon-001}\mathcal{P}_{f^{\ast}}(z)=\frac{1}{2\pi}\int_{0}^{2\pi}P(z,e^{it})f(e^{it})dt,
\quad
\mathcal{P}_{\varphi_{1}}(z)=\frac{1}{2\pi}\int_{0}^{2\pi}P(z,e^{it})\varphi_{1}(e^{it})dt,
\ee
\be\label{mon-002} \varphi_{1}(e^{it})=\varphi(e^{it})e^{-it}\quad \mbox{and}\quad   G[g](z)=\frac{1}{2\pi} \int_{\mathbb{D}} g(w)G(z,w)dA(w).
\ee
Here $dA(w)$ denotes the Lebesgue area measure in $\mathbb{D}$.

The solvability of the non-homogeneous biharmonic equations has also been studied in \cite{MM09}.

Let us recall the following version of the boundary Schwarz lemma of analytic functions, which was proved in \cite{LT15}.

\begin{Thm}{\rm (\cite[Theorem~$1.1'$]{LT15})}\label{MWZ} Suppose that $f$ is an analytic function from $\mathbb{D}$ into itself. If $f(0)=0$ and $f$ is analytic at $z=\alpha\in \mathbb{T}$ with $f(\alpha)=\beta\in \mathbb{T}$, then
\begin{enumerate}
\item $\overline{\beta}f'(\alpha)\alpha \geq 1$.
\item $\overline{\beta}f'(\alpha)\alpha = 1$ if and only if $f(z)\equiv e^{i\theta} z$, where $e^{i\theta}=\beta \alpha^{-1}$ and $\theta\in \R$.
\end{enumerate}
\end{Thm}

This useful result has attracted much attention and has been
generalized in various forms (see, e.g., \cite{CK,Zhu18-fil}). Recently, Wang et. al. obtained a boundary Schwarz lemma for
the solutions to Poisson's equation (\cite{WZ18}).
By analogy with the studies in \cite{WZ18}, we discuss the boundary Schwarz lemma for the functions with the form \eqref{eq-ch-3}. Our result is as follows. Note that a different form of the boundary Schwarz lemma
for functions with the form \eqref{eq-ch-3} was proved in \cite{CLW}.

\begin{theorem}\label{bsl-2}
Suppose $f\in \mathcal{C}^4(\mathbb{D})$ and $g\in \mathcal{C}(\overline{\mathbb{D}})$ satisfy the following equations:
$$\left\{\begin{array}{ll}
	\Delta (\Delta f)=g & \mbox{ in } \mathbb{D},\\
	f_{\overline{z}}=\varphi & \mbox{ on } \mathbb{T},\\
	f=f^\ast & \mbox{ on } \mathbb{T},
	\end{array}\right.
$$
where $\varphi\in \mathcal{C}(\mathbb{T})$, $f^\ast \in \mathcal{C}(\overline{\mathbb{D}})$, $f^\ast $ is analytic in $\mathbb{D}$ and $f(\mathbb{D})\subset \mathbb{D}$. If $f$ is differentiable at $z=\alpha\in \mathbb{T}$, $f(\alpha)=\beta\in \mathbb{T}$ and $f(0)=0$, then
\be\label{wen-1} \real [\overline{\beta}(f_z(\alpha)\alpha+f_{\overline{z}}(\alpha)\overline{\alpha})]\ge
\frac{2}{\pi}-3\|\mathcal{P}_{\varphi_1}\|_\infty-\frac{1}{64} \|g\|_{\infty},
\ee
where $\mathcal{P}_{\varphi_1}$  and $\varphi_1$ are defined in \eqref{mon-001} and  \eqref{mon-002}, respectively.

In particular, when $\|\mathcal{P}_{\varphi_1}\|_\infty=\|g\|_{\infty}=0$, the following inequality is sharp:
\be\label{mana-1}\real [\overline{\beta}(f_z(\alpha)\alpha+f_{\overline{z}}(\alpha)\overline{\alpha})]\ge
\frac{2}{\pi}.\ee
\end{theorem}

We have the following two remarks.
\begin{enumerate}
\item
For analytic functions, the value of $\overline{\beta}f'(\alpha)\alpha$ in Theorem A is
a real number. However, this is not true for the case of the solutions to the equation
\eqref{eq-ch-1.0} (see Example~\ref{existence} below). Hence, in Theorem \ref{bsl-2}, we consider the real part of the quantity $\overline{\beta}(f_z(\alpha)\alpha+f_{\overline{z}}(\alpha)\overline{\alpha})$.
\item The obtained lower bound for the quantity $\real [\overline{\beta}(f_z(\alpha)\alpha+f_{\overline{z}}(\alpha)\overline{\alpha})]$ in \eqref{wen-1} is always positive for all $\varphi_1$ and $g$ with $(\|\mathcal{P}_{\varphi_1}\|_{\infty}, \|g\|_{\infty})\in \{(x,y):\; x\geq 0,\; y\geq 0,\; 3x+\frac{1}{64}y<\frac{2}{\pi}\}$.
\end{enumerate}

\section{Proof of Theorem \ref{bsl-2}}
We start with the following lemma.

\begin{lemma}\label{bsl-1}
Suppose $\mathfrak{g}\in \mathcal{C}(\overline{\mathbb{D}})$ and $h \in \mathcal{C}^4(\mathbb{D})$ satisfy the following equations:
$$\left\{\begin{array}{ll}
\Delta (\Delta h)=\mathfrak{g} & \mbox{ in } \mathbb{D},\\
h_{\overline{z}}=\psi & \mbox{ on } \mathbb{T},\\
h=h^\ast & \mbox{ on } \mathbb{T},
\end{array}\right.
$$
where $\psi \in \mathcal{C}(\mathbb{T})$, $h^\ast \in \mathcal{C}(\overline{\mathbb{D}})$, $h^\ast $ is analytic in $\mathbb{D}$ and $h(\mathbb{D})\subset \mathbb{D}$. If $h$ is differentiable at $z=1$, $h(1)=1$ and $h(0)=0$, then
\[ \real [h_z(1)+h_{\overline{z}}(1)]\ge
\frac{2}{\pi}-3\|\mathcal{P}_{\psi_1}\|_\infty-\frac{1}{64} \|\mathfrak{g}\|_{\infty},
\]
where $\psi_1(e^{it})=\psi(e^{it}) e^{-it}$.

In particular, when $\|\mathcal{P}_{\psi_1}\|_\infty=\|\mathfrak{g}\|_{\infty}=0$, the following inequality is sharp:
\be\label{mana-2}\real \left[h_z(1)+h_{\overline{z}}(1)\right]\ge
\frac{2}{\pi}.\ee
\end{lemma}
\begin{proof}
The assumptions of the lemma ensure that $h$ has the form \eqref{eq-ch-3}, i.e.,
$$
h(z)=\mathcal{P}_{h^{\ast}}(z)+\frac{1}{2\pi}(1-|z|^{2})\int_{0}^{2\pi}h^{\ast}(e^{it})\frac{\overline{z}e^{it}}{(1-\overline{z}e^{it})^{2}}dt-(1-|z|^{2})
\mathcal{P}_{\psi_{1}}(z)-\frac{1}{8}G[\mathfrak{g}](z).
$$

Since the analyticity of $h^\ast$ in $\mathbb{D}$ gives
\be\label{eq-anal}
\frac{1}{2\pi}\int_{0}^{2\pi}\overline{z}e^{it}h^{\ast}(e^{it})\frac{1-|z|^{2}}{(1-\overline{z}e^{it})^{2}}dt=0,
\ee
we obtain that
\begin{eqnarray}\label{wed-001}
|h(z)| &=& \left| \mathcal{P}_{h^\ast}(z)-(1-|z|^2) \mathcal{P}_{\psi_1}(z)-\frac{1}{8} G[\mathfrak{g}](z)\right|\\ \nonumber
&\le & \left| \mathcal{P}_{h^\ast}(z) -\frac{1-|z|^{2}}{1+|z|^{2}}\mathcal{P}_{h^{\ast}}(0)\right|+
(1-|z|^2)\left|\mathcal{P}_{\psi_1}(z)- \frac{1-|z|^{2}}{1+|z|^{2}}\mathcal{P}_{\psi_{1}}(0)\right|\\ \nonumber
&& +\frac{1-|z|^2}{1+|z|^2} \left(\left|\mathcal{P}_{h^{\ast}}(0)-\mathcal{P}_{\psi_1}(0)\right|+|z|^2\left|\mathcal{P}_{\psi_1}(0)\right|\right)+\left|\frac{1}{8} G[\mathfrak{g}](z)\right|.
\end{eqnarray}

By the proof of Theorem $1.1$ in \cite{CLW}, we have the following estimates:
$$ \left| \mathcal{P}_{h^\ast}(z) -\frac{1-|z|^{2}}{1+|z|^{2}}\mathcal{P}_{h^{\ast}}(0)\right|\le
\frac{4}{\pi} \|\mathcal{P}_{h^{\ast}}\|_{\infty}\arctan|z|,
$$
$$\left|\mathcal{P}_{\psi_1}(z)- \frac{1-|z|^{2}}{1+|z|^{2}}\mathcal{P}_{\psi_{1}}(0)\right| \le
\frac{4}{\pi}\|\mathcal{P}_{\psi_{1}}\|_{\infty}\arctan|z|
$$
and
$$
\left|G[\mathfrak{g}](z)\right|\le \frac{1}{8}\|\mathfrak{g}\|_\infty (1-|z|^2)^2.
$$

Moreover, it follows from the assumption $h(0)=0$ that
$$ \mathcal{P}_{h^{\ast}}(0)-\mathcal{P}_{\psi_1}(0)=\frac{1}{8} G[\mathfrak{g}](0),
$$ and so, we get
$$
|\mathcal{P}_{h^{\ast}}(0)-\mathcal{P}_{\psi_1}(0)|\le \frac{1}{64} \|\mathfrak{g}\|_{\infty}.
$$

Based on the above estimates, together with the fact $\|\mathcal{P}_{h^\ast}\|_{\infty}\le 1$, the inequality \eqref{wed-001} is changed into the following form:
\begin{eqnarray}
\label{eqn-f-bound} |h(z)| &\le&  \frac{4}{\pi}
\arctan|z|+ \frac{1-|z|^2}{1+|z|^2}
\left(\frac{1}{64} \|\mathfrak{g}\|_{\infty}+|z|^2 \|\mathcal{P}_{\psi_1}\|_\infty \right)
\\
&&\nonumber +\frac{4}{\pi}\|\mathcal{P}_{\psi_{1}}\|_{\infty}(1-|z|^{2})\arctan|z| +\frac{1}{64}\|\mathfrak{g}\|_\infty (1-|z|^2)^2\\
&=:&\nonumber M(|z|).
\end{eqnarray}

Since $h$ is differentiable at $z=1$, we have
\begin{equation}\nonumber
h(z)=1+h_z(1)(z-1)+h_{\overline{z}}(1)(\overline{z}-1)+o(|z-1|),
\end{equation}
where $o(x)$ means a function with $\lim_{x\to 0} o(x)/x=0$. Then we deduce from \eqref{eqn-f-bound} that
$$ 2 \real [h_z(1)(1-z)+h_{\overline{z}}(1)(1-\overline{z})]\ge 1-M^2(|z|)-o(|z-1|).
$$

By letting $z=r\in (0,1)$ and $r\to 1^-$, we get
$$\real [h_z(1)+h_{\overline{z}}(1)]\ge  \lim_{r\to 1^-} M'(r)=
\frac{2}{\pi}-3\|\mathcal{P}_{\psi_1}\|_\infty-\frac{1}{64} \|\mathfrak{g}\|_{\infty}.
$$

To finish the proof of the lemma, it remains to check the sharpness of the inequality
\eqref{mana-2}. For this, we borrow the following function from \cite[Page 127]{ABR}:
\be\label{mana-3}
\mathfrak{h}(z)=\left\{\begin{array}{ll}
\frac{2}{\pi} \arctan \frac{z+\overline{z}}{1-|z|^2}\; \mbox{ if }\; z\in \mathbb{D},\\
\;\;\;\;\;\;\;\;\;\;\;1\; \;\;\;\;\;\;\;\,\ \mbox{ if }\; z\in\mathbb{T}.
\end{array}\right.
\ee
It can be seen that $\mathfrak{h}$ is harmonic in $\mathbb{D}$ with $\mathfrak{h}(0)=0$ and $\mathfrak{h}(1)=1$.

Since
\be\label{lem-sh} \mathfrak{h}_z(z)=\frac{2}{\pi} \frac{1+\overline{z}^2}{(1-|z|^2)^2+(z+\overline{z})^2}\;\; \mbox { and }\;\;
\mathfrak{h}_{\overline{z}}(z)= \frac{2}{\pi} \frac{1+{z}^2}{(1-|z|^2)^2+(z+\overline{z})^2},
\ee we know that both $\mathfrak{h}_z$ and $\mathfrak{h}_{\overline{z}}$ are continuous at $z=1$. This
guarantees the differentiability of $\mathfrak{h}$ at this point.

Let
$$\mathfrak{h}^\ast(z)=1
$$ in $\mathbb{\overline{D}}.$
It is clear that $\mathfrak{h}^\ast$ is analytic in $\mathbb{D}$ and $\mathfrak{h}=\mathfrak{h}^\ast$ on $\mathbb{T}$. Further, the harmonicity of $\mathfrak{h}$ in $\mathbb{D}$, together with \cite[(1.5)]{CLW} and \eqref{eq-anal}, ensures that $$\mathcal{P}_{\psi_1}=0.$$  Since \eqref{lem-sh} leads to
$$ \real \left[\mathfrak{h}_z(1)+\mathfrak{h}_{\overline{z}}(1)\right]=\frac{2}{\pi},
$$ we see that $\mathfrak{h}$ is our needed extremal function for the sharpness of \eqref{mana-2}.
The proof of the lemma is complete.
\end{proof}

\begin{proof}[\bf Proof of Theorem \ref{bsl-2}]
Let $$h(z)=\overline{\beta} f(\alpha z)$$ in $\mathbb{D}$, and let
$$\mathfrak{g}(z)=\overline{\beta}g(\alpha z)$$ in $\mathbb{D}$,
$$\psi(\xi)=\overline{\beta}\overline{\alpha}\varphi(\alpha \xi)$$ on $\mathbb{T}$, and
$$h^*(z)=\overline{\beta}f^*(\alpha z)$$ in $\overline{\mathbb{D}}$.
Then we know from Lemma \ref{bsl-1} that
$$ \real [h_z(1)+h_{\overline{z}}(1)]\ge
\frac{2}{\pi}-3\|\mathcal{P}_{\varphi_1}\|_\infty-\frac{1}{64} \|g\|_{\infty},
$$
from which the inequality \eqref{wen-1} in Theorem \ref{bsl-2} follows since
$$\real [\overline{\beta}(f_z(\alpha)\alpha+f_{\overline{z}}(\alpha)\overline{\alpha})]=\real [h_z(1)+h_{\overline{z}}(1)].
$$

 The inequality \eqref{mana-1} is obvious. For its sharpness, let
$$\mathfrak{f}(z)=\frac{2 \beta}{\pi} \arctan \frac{\overline{\alpha}z+\alpha \overline{z}}{1-|z|^2}
$$
in $\mathbb{D}$. Then $$\mathfrak{f}(z) =\beta \mathfrak{h}(\overline{\alpha} z),$$
where the function $\mathfrak{h}$ is defined in \eqref{mana-3}.  By the discussions on the sharpness of the inequality \eqref{mana-2} in the proof of Lemma \ref{bsl-1},
we see that $\mathfrak{f}$ is the needed function for the sharpness of the inequality \eqref{mana-1}.
Now, the theorem is proved.
\end{proof}

\section{An example}
In this section, we construct an example to show that, in Theorem \ref{bsl-2}, it is reasonable for us to consider the real part of the quantity $\overline{\beta}(f_z(\alpha)\alpha+f_{\overline{z}}(\alpha)\overline{\alpha})$.

\begin{example}\label{existence}
	Assume that
$$g(z)=32 M\left[2-3i(z^2+\overline{z}^2)\right]$$ and
	$$f(z)=	(1-M)z^2+\frac{Mi}{4} (1-|z|^4)(z^2+\overline{z}^2)+M|z|^4$$ in $\overline{\mathbb{D}}$,
	where $0<M<\frac{2}{35\pi}\sqrt{5}(3-\sqrt{2})$.
		Then
\begin{enumerate}
\item
$f$ and $g$ satisfy the following non-homogeneous biharmonic equation $$\Delta^2 f=g,$$  and all other assumptions in Theorem \ref{bsl-2} with $\alpha=\beta=1$;
\item
${\rm Re} \big(f_z(1)+f_{\overline{z}}(1)\big)=2(1+M)$, $\|\mathcal{P}_{\varphi_1}\|_\infty=\sqrt{5}M$, $\|g\|_{\infty}=64\sqrt{10}M$ and
$${\rm Im} \big(f_z(1)+f_{\overline{z}}(1)\big)=-2M\not=0,$$
where $\varphi_{1}(\zeta)=\frac{M}{2} \left(4-i (\zeta^2+\overline{\zeta}^2)\right )$ on $\mathbb{T}$.
\end{enumerate}	
\end{example}

\begin{proof}  Elementary computations yield
	\be\label{mw-1} f_z(z)=2(1-M)z+\frac{Mi}{2} \left[z(1-|z|^4)-z\overline{z}^2(z^2+\overline{z}^2)\right]+2Mz\overline{z}^2,
	\ee
		\be\label{mw-2}f_{\overline{z}}(z)=\frac{Mi}{2} \left[\overline{z}(1-|z|^4)-z^2\overline{z}(z^2+\overline{z}^2)\right]+
	2Mz^2\overline{z}
	\ee
	and $$\Delta^2 f=g.$$

Obviously, $f(0)=0$ and $f(1)=1$. Let
$$\varphi(\zeta)=\frac{M}{2}\zeta \left(4-i (\zeta^2+\overline{\zeta}^2)\right )$$ on $\mathbb{T}$,
and
$$f^\ast(z)=(1-M)z^2+M$$ in $\overline{\mathbb{D}}.$ Then $$f_{\overline{z}}=\varphi$$  on $\mathbb{T}$, $f^*$ is analytic in $\mathbb{D}$, and
$$f^*=f$$ on $\mathbb{T}$.

	Since for $z\in \mathbb{D}$,
	$$ |f(z)|\le |z|^2 \left(1-\frac{M}{2} (1-|z|^2)^2 \right)<1,
	$$
	we see that $f(\mathbb{D})\subset \mathbb{D}$.
	
Moreover, the differentiability of $f$ at $z=1$ can be seen from the continuity of
	its partial derivatives (cf. \eqref{mw-1} and \eqref{mw-2}).
	Now, we have proved that the first conclusion of the example is true.
	
The equalities
$$ {\rm Re} \big(f_z(1)+f_{\overline{z}}(1)\big)=2(1+M)\;\;\mbox{and}\;\; {\rm Im} [f_z(1)+ f_{\overline{z}}(1)]=-2M\not=0
	$$ easily follow from \eqref{mw-1} and \eqref{mw-2} and elementary computations give
$$\|\mathcal{P}_{\varphi_1}\|_\infty=\max_{z\in \overline{\mathbb{D}}}\left\{\frac{M}{2} \left |4-i (z^2+\overline{z}^2)\right |\right\}=\sqrt{5}M$$
and
$$\|g\|_{\infty}=\max_{z\in \overline{\mathbb{D}}}\left\{32 M\left |2-3i(z^2+\overline{z}^2)\right |\right\}=64\sqrt{10}M.$$
Hence the second conclusion of the example is true too, and so, the proof of the example is complete.
\end{proof}

\begin{remark}
The purpose to add the condition $0<M<\frac{2}{35\pi}\sqrt{5}(3-\sqrt{2})$ in Example \ref{existence} is to guarantee that $$3\|\mathcal{P}_{\varphi_1}\|_\infty+\frac{1}{64} \|g\|_{\infty}<\frac{2}{\pi},$$ i.e., the quantity $\frac{2}{\pi}-3\|\mathcal{P}_{\varphi_1}\|_\infty-\frac{1}{64} \|g\|_{\infty}$ is positive.
\end{remark}

\vspace*{5mm}
\noindent {\bf Acknowledgments}.
The research was partly supported by NSFs of China (Nos 11571216, 11671127 and 11720101003) and STU SRFT.
The third author was supported by NSF of Fujian Province (No. 2016J01020)
and the Promotion Program for Young and Middle-aged Teachers in Science and Technology Research of Huaqiao University (ZQN-PY402).

\end{document}